\documentclass[2pt]{amsart}
\usepackage{amssymb,amsfonts, amsthm}
\usepackage{amscd, euscript, MnSymbol}
\usepackage{amscd, euscript}
\usepackage{pdfsync}
\usepackage{enumerate}
\usepackage{epsfig,graphicx,graphics} 
\usepackage{pdfsync}

\newtheorem{theorem}{Theorem}[section]
\newtheorem{lemma}[theorem]{Lemma}
\newtheorem{proposition}[theorem]{Proposition}
\newtheorem{corollary}[theorem]{Corollary}
\theoremstyle{definition}

\newtheorem{problem}[theorem]{{\sc Problem}}

\newtheorem{claim}{\noindent {\bf Claim}}

\newcommand{\R}{\mathbb{R}}

\newcommand{\N}{\mathbb{N}}

\newcommand{\Z}{\mathbb{Z}}

\let\rho=\varrho
\DeclareMathOperator{\Eqv}{Eqv}

\DeclareMathOperator{\Pol}{Pol}
\DeclareMathOperator{\Cong}{Cong}

\title[Operations preserving equivalence relations]{Operations preserving equivalence relations}

\author[M.Pouzet] {Maurice Pouzet}
\address{Univ. Lyon, University Claude-Bernard  Lyon1, UMR 5208, Institut Camille Jordan, 
43, Bd. du 11 Novembre 1918,
69622 Villeurbanne, France et Department of Mathematics and Statistics, The University of Calgary, Calgary, Alberta, Canada}
\email{pouzet@univ-lyon1.fr}

\date{\today}

\begin{document}

\begin{abstract}
  In \cite{cgg1}, C\'egielski, Grigorieff and Guessarian characterized unary self-maps on the set $\Z$ of integers which preserve all congruences of the additive group.  In this note, we propose a shorter and straigthforward proof. We replace this result in the frame of universal algebra. 
%
\end{abstract}
\maketitle
\section{Introduction}
A {\it binary relation} on a set $A$ is a  subset $\rho$ of the cartesian product $A\times A$. We write $x\rho y$ instead of $(x,y)\in \rho$. If $\rho$ is an equivalence relation, we also write $x\equiv y (\rho)$.

A map $f:A\rightarrow A$ \emph{preserves} $\rho$ if \begin{equation}
x \rho y \Rightarrow f(x)\rho f(y) \label{eq:1}
\end{equation}
for all $x,y\in A$.

For a  non-negative integer $n$, a map $f:A^n\rightarrow A$ is called an \emph{operation} on $A$. Such an operation \emph{preserves} $\rho$ if \begin{equation}
 f(x)\rho f(y) \label{eq:2}
\end{equation}
 for every $x:= (x_1,\dots x_n), y:=(y_1,\dots y_n)\in A^n$ such that $x_i\rho y_i$ for all $i=1,n$. 
 
 A pair made a set $A$ and a  collection $F$ of  operations on $A$ is called an \emph{algebra}. Equivalence relations preserved by all members  of $F$ are called \emph{congruences}. The study of the relationship between the set of congruences of an algebra and the set of maps which preserve all congruences is one of the goals of universal algebra. 

 For example, if the algebra is the set $\Z$ of relative integers equipped with the group operation $+$, the congruences of this algebra are the ordinary congruences: each one, say $\equiv_r$,   is determined  by a non-negative integer $r$ and is defined by $x\equiv_r y$ if $x-y$ is a multiple of $r$. On the opposite, the determination  of the operations preserving all the congruences of $(\Z, +)$ is a more serious task. 

 This is handled for unary functions in \cite{cgg1} by  C\'egielski, Grigorieff and Guessarian (CGC). Their  description is given in terms of Newton expansion. 
 The proof is by no means trivial.
 In this note we  propose a shorter and straigthforward proof.  While we were writing this paper, C.Delhomm\'e \cite{delhomme1} found a five lines proof of the main argument (Lemma \ref{wasdifficult}); we reproduce it with his permission. In the last section we replace the CGC result in the frame of universal algebra.

 \section {Maps preserving the congruences of $\Z$}
The set $\mathcal C$ of maps $f:\Z \rightarrow \Z$ which preserve all congruences on $\Z$ is \emph{locally closed}, meaning that  $f\ \mathcal C$  iff for every finite subset $A$ of $\Z$, (in fact, every $2$-element subset of $\Z$) $f$ coincides on $A$ with some $g \in  \mathcal C$ (in topological terms, $\mathcal C$ is a closed subset of the topological space  $\Z^{\Z}$ of maps $f:\Z \rightarrow \Z$ equipped with the pointwise convergence topology, the topology on $\Z$ being discrete). The set $\mathcal C$  contains all polynomials with integer coefficients. But it contains others (e.g.  the polynomial  $g(x):= \frac{x^2(x-1)^2}{2}$  is a congruence preserving map on $\Z$) and other maps than polynomials. 
 
 Let $n$ be  a non-negative integer, let $lcm(n):= 1$ if $n=0$, otherwise  let $lcm ({n})$ be the least common multiple of $1, \dots n$, i.e. $lcm({n}):= lcm \{1, \dots, n\}$. If $X$ is an indeterminate (as well as a number) we set $X^{\underline 0}:= 1$, $X^{\underline 1}:=X$, $X^{\underline n}:= X\cdot (X-1)\cdot \dots \cdot (X-n+1)$. The \emph{binomial polynomial}  is  ${X \choose n}:= \frac{X^{\underline n}}{n!}$. 
 CGG  show  that:

  \begin{theorem}\label{main}
A map  $f: \Z\rightarrow \Z$  preserves all congruences iff this is an infinite  sum 
 
 $$f(x)= \sum_{n=0, \infty} a_{n}\cdot P_{n}(x)$$ where each $a_n$ is an integer multiple of $lcm ({n})$ and  $P_{n}$ the polynomial equal to 
 ${X+k\choose 2k}$ if $n=2k$  and equal to ${X+k\choose 2k+1}$ if $n= 2k+1$. 
  
  \end{theorem}  
  Two special cases of this result  are:
   
 \begin{enumerate}
  \item Polynomial  functions of the form  ${lcm(n)}\cdot {x \choose n}$ preserve all congruences; 
   \item Every polynomial  function which preserves all congruences  is a finite linear sum with integer coefficients  of these polynomials; 
   \end{enumerate}

Using a result of Kaarli \cite{kaarli}, based on the Chinese remainder theorem, it follows from Theorem \ref{main}: 

\begin{theorem} \label {thm:extension}Every map $f$ from a finite subset $A$ of $\Z$ and with values in $\Z$ which preserves all congruences extends to  a polynomial function  preserving all congruences. 
\end{theorem}
 
\begin{corollary}
The set $ \mathcal C$ of maps  $f: \Z\rightarrow \Z$ which preserve all congruences is the local closure of the set of polynomials preserving all congruences.
\end{corollary}
Being closed  in  the set $\Z^{\Z}$ of all maps $f :\Z\rightarrow \Z$ endowed with the pointwise convergence topology, the set $\mathcal C$ is a Baire subset of $\Z^{\Z}$.  Hence, it is uncountable (apply Theorem \ref{main}, or observe that it has no isolated point and apply Baire theorem). In particular,  it contains   functions which are not polynomials. A striking example using Bessel functions is given in CGG's paper.    
   


The  description  of polynomial  functions with integer values  was given  by Polya in 1915 (cf Theorem 22 page 794 in Bhargava \cite{bhargava}). It is instructive. Let us look at it. 
\begin{lemma}\label{integervalues} Polynomial  functions from $\Z$ to $\Z$ are finite linear sums with integer coefficients of polynomial functions of the form ${x \choose k}$. 
\end{lemma}
\begin{proof}
Let $P$ be a polynomial of degree $n$ over the reals. Since the ${X\choose k}$, $k\in \N$,  have different degrees, they form a basis, hence 
$$P:= \lambda_0+ \dots+ \lambda_k \cdot  {X\choose k}+ \dots+ \lambda_n \cdot {X\choose n}$$ for some reals $\lambda_0, \dots, \lambda_n$. 

Since $[{X\choose k}](X=m)$ is a binomial coefficient (for $k\leq m$), every linear combination with integer coefficients of these polynomials takes integer values. Thus, if the $\lambda_k$'s are integers, $P$ takes integer values. Conversely, suppose  
 that the values of $P$ are integers  for $X:= 0, \dots, n$. A trivial recurrence on the degree  will show that the coefficients are integers. 
Indeed, let $$Q:= \lambda_0+ \dots+ \lambda_k \cdot  {X\choose k}+ \dots+ \lambda_{n-1}\cdot {X \choose n-1}.$$
Since  $Q(k)=P(k)$ for all $k\leq n-1$, each $Q(k)$ is an integer. Hence  induction applies to $Q$ and yields  that all $\lambda_0, \dots, \lambda_{n-1}$ are integers. Now, $P(n)= Q(n)+ 
\lambda_{n}\cdot  [{X\choose n}](X=n )$. Since $\lambda_0, \dots, \lambda_{n-1}$ are integers, $Q(n)$ is an integer; since  $[{X\choose n}](X= n)=1$, it follows that $\lambda_n$ is an integer.  This proves our affirmation about the integrality of the coefficients. 

 The lemma is proved. 

\end{proof}

One can say a bit more:
\begin{lemma}\label{onepolynomialextension} $(a)$ Every map $f$ from a non-empty finite subset $A$ of $\Z$ and values in $\Z$ extends to a polynomial function with integer values and degree at  most $n$ where $n+1$ is the cardinality of the smallest interval containing $A$. $(b)$ For every map  $f:\Z :\rightarrow \Z$  there are integer coefficients $a_n, n\in \N$, such that  $$f(x)= \sum_{n=0, \infty} a_{n}\cdot P_{n}(x)$$ for every  $x\in \Z$, where $P_{n}$ the polynomial equal to 
 ${X+k\choose 2k}$ if $n=2k$  and equal to ${X+k\choose 2k+1}$ if $n= 2k+1$.
\end{lemma}

\begin{proof}
\noindent $(a)$. Beware, Lagrange approximation will not do; for an example, in order to extend a map defined on a $2$-element subset, we may need a polynomial of large degree. First, we show by induction on $n$ that any $f$ defined on $\{0, \dots n\}$ extends to a polynomial of degree at most $n$. If $n = 0$, any constant polynomial equal to $f(0)$ will do. Suppose $n> 0$. Via the induction hypothesis, there is some polynomial $Q:=  \lambda_0+ \dots+ \lambda_k \cdot  {X\choose k}+ \dots+ \lambda_{n-1}\cdot {X\choose {n-1}}$ extending $f$ on $\{0, \dots n-1\}$. Let $P:= Q+\lambda_{n}\cdot {X\choose n}$. We have $P(k)= Q(k)=f(k)$ for all $k<n$ and $P(n)= Q(n)+\lambda_{n}$. If we set $\lambda_{n}:= f(n)-Q(n)$ we have $P(n)=f(n)$. Hence $P$ extends $f$, proving that the property holds for all $n$. Now, if $f$ is defined on $A$, let $\overline A$ be the least interval containing $A$. Extend $f$ to $\overline f$ defined on $\overline A$ and taking  integer values. Then translate $\overline A$ to $\{0, \dots, n\}$ and apply the previous case.

\noindent $(b)$. We define a sequence of intervals $A_n$, $n\in \N$ of $\Z$ where $A_{0}:= \emptyset$ and for $k \in \N$,   $A_{2k}:= \{-k, \dots, k-1\}$ and $A_{2k+1}:= \{-k, \dots, k\}$. Let  $(b_n)_{n\in \N}$ be the sequence of elements of $\Z$ defined by $b_{2k}:= k$, $b_{2k+1}:= -k-1$.  Trivially, this sequence exhaust $\Z$,  $P_n$ is identically $0$ on $A_n$ and  $\{b_n\}= A_{n+1}\setminus A_n$.  Furthermore, $P_n(b_n)=(-1)^{n}$ (An other choice for  $P_n$ would have been simpler).  We set  $a_0:= f(b_0)=f(0):= \overline f(b_0)$. Let $n>0$. Suppose $a_i$ defined for $i<n$. Choose $a_n$ such that  $a_n\cdot  P_n (b_n)=f(b_n)-\sum_{i<n} a_i\cdot P_i(b_n)$.  Then, $\overline f(b_n)= \sum_{i<n} a_i\cdot P_i(b_n)+ a_n\cdot P_n(b_n)$. Since $P_{n+k}(b_n)=0$ for $k\geq 1$, we have  $\overline f(b_n)= \sum_{i:= 0, \infty} a_i\cdot P_i(b_n)$, hence  $\overline f(x)= \sum_{n=0, \infty} a_{n}\cdot P_{n}(x)$ as claimed.   
For an example, $\overline f(b_0)= a_0$, $\overline f(b_1)= a_0-a_1$, $\overline f(b_2)= a_0+a_1+a_2$. 

\end{proof}

The crucial fact in Lemma \ref{integervalues} was that the polynomial ${X\choose k}$ takes integer values. In order to prove Theorem \ref {main} we need to prove that the polynomial $lcm (k)\cdot {X\choose k}$ preserves all congruences.

\begin{lemma}\label{smalltrick}  Let $f(x):= {\lambda_k}\cdot {x\choose k}$. If $f$ preserves the congruences $\equiv_{i}$ for all $i:=0, \dots, k$ then $\lambda_k$ is a multiple of $lcm(k)$.
\end{lemma} 
\begin{proof}
 For $i:=0, 1, \dots, k-1$, we have $f(i)=0$. If $f$ preserves $\equiv_{k-i}$, $f(k)=f(k)-f(i)$ is a multiple of $k-i$, hence  
 $f(k)$ is a multiple of $k, k-1, \dots, 1$. Since $f(k)= \lambda_k$, the result follows. 
\end{proof}

\begin{lemma}\label{wasdifficult} Let $n$ be  a non-negative integer and  $f_n(x):= {lcm(n)}\cdot {x\choose n}$. Then $f$ preserves all congruences. \end{lemma} 

\begin{proof}
\begin{claim}\label {fnk}
$\frac{f_n(k)}{k}={\frac{lcm(n)}{k}}\cdot {k\choose n}$ is an integer for every  $k\in \Z\setminus \{0\}$. 
\end{claim}
\noindent {\bf Proof of Claim \ref{fnk}.}
Suppose  $k\geq n$. 
We have ${\frac{lcm(n)}{k}}\cdot {k\choose n}:= \frac{lcm(n)}{n} \cdot \frac{(k-1)\cdots (k-n+1)} {n-1!}=\frac {lcm(n)}{n}\cdot {{k-1}\choose {n-1}}$. By definition $lcm(n)$ is a multiple of $n$ and the binomial coefficient ${{k-1}\choose {n-1}}$ is an integer. This yields the result.  If $k\in [0, n-1]$ then $f_n(k)=0$ hence the property holds. If $k<0$ set $k' := -k$. Then ${\frac{lcm(n)}{k}}\cdot {k\choose n}= \frac{lcm(n)}{n} \cdot \frac{(-k'-1)\cdots (-k'-n+1)} {n-1!} =\frac{lcm(n)}{n} \cdot (-1)^{n-1}\cdot \frac{(k'+1)\cdots (k'+n-1)} {n-1!}= \frac{lcm(n)}{n} \cdot (-1)^{n-1}\cdot {{k'+n-1}\choose {n-1}}$. Again, this number is an integer. \hfill $\Box$
\begin{claim}\label {fn}
$k$ divides $f_n(x+k)-f_n(x)$ for all  $n\in \N$, $k\in \Z\setminus \{0\}$ and $x\in \Z$. 
\end{claim}
\noindent {\bf Proof of Claim \ref{fn}.}

Suppose first that  $x\in \N$.

We use induction on $n$ and on $x$. 

$\bullet$ $n:=0$.  In this case, $f_n$ is identically $0$, hence  $f_n(x+k)-f_n(x)=0$  and the property holds. 

$\bullet$ $n=1$.  Then   $lcm(1):=1$, $f_n(x+k)=x+k$ hence $f_n(x+k)-f_n(x)=k$ and the property holds. 

$\bullet$ $n>1$. Suppose that the property holds for $n-1$. We proceed by induction on $x$. 

$\bullet$ $x=0$.  We have  $f_n(x+k)-f_n(x)= f_n(k)-f_n(0)= f_n(k)$. According to Claim \ref{fnk} this quantity is divisible by $k$ for every non-zero $k$.   

$\bullet$ $x>0$. 

We use the following form of Pascal identity:

$$ {X \choose n}= {{X -1}\choose n}+ {{X-1} \choose {n-1}}.$$
For $X:=x+ k$ this yields:
$$ {x+k \choose n}= {{x+k -1}\choose n}+ {{x+k-1} \choose {n-1}}.$$
 
For $X:= x$ this yields: 
$$ {x \choose n}= {{x -1}\choose n}+ {{x-1} \choose {n-1}}.$$
  
Hence, via a substraction, we have: 
$$\frac{1}{k}\cdot (f_n(x+k)-f_n(x))=\frac{1}{k}\cdot (f_n(x+k-1)-f_n(x-1))+\frac{a}{k}\cdot (f_{n-1}(x+k-1)-f_{n-1}(x-1)).$$ where $a.lcm(n-1)=lcm(n)$. 

 Via the induction on $x$,  the first term of the sum is an integer, whereas the second is an integer via the induction on $n$.

To complete the proof, let $x<0$. Set $x':= -x$. Then 

$$f_n(x+k)-f_n(x)=f_n(-x'+k)-f_n(-x')= 
 (-1)^{n}\cdot ( f_{n} (x'-k+n-1)-f_{n}(x'+n-1).$$
 
 Since $x'>0$ this quantity is   divisible by $-k$ that is $k$. 
 \hfill $\Box$
 
 With Claim \ref {fn} the conclusion of the  lemma follows. 
 
\end{proof}

A much shorter proof was discovered by C.Delhomm\'e \cite{delhomme1}.  He obtains  the fact that 
 $f_n(x+k)-f_n(x)$ is divisible  by $k$ from the equalities: 
 
 \begin{equation} 
 {{x+k}\choose n}-{x \choose n}= \sum_{i=1, \dots, n}{x\choose {n-i} }\cdot{k\choose i}= \sum_{i=1, \dots n}{x\choose {n-i}} \frac{k}{i}{{k-1}\choose {i-1}}.
 \end{equation}
 
 Indeed, $\frac{lcm(n)}{i}$ is an integer for every $i=1,\dots,  n$.  
To prove that the first  equality holds, it suffices to check that its holds for infinitely many values of $x$. So suppose $x, k\in \N$. In this case, the left hand side  counts  the number of $n$-element subsets $Z$ of a $x+ k$-element set union of  two disjoints set $X$ and $K$ of size $x$ and $k$, each $Z$ meeting $K$.  Dividing this collection of  subsets according to the size of their intersection with $K$ yields the right hand size of this equality.

\begin{lemma} \label{preservecongruence}Polynomial  functions from $\Z$ to $\Z$ which preserve all congruences are finite linear sums with integer coefficients of polynomial functions of the form $lcm(k)\cdot {x\choose k}$. 
\end{lemma}
\begin{proof}
We adapt the proof of Lemma \ref{integervalues}. 
Let $P$ be  a polynomial from $\Z$ to $\Z$. According to Lemma \ref{integervalues}
$$P:= \lambda_0+ \dots+ \lambda_k \cdot  {X\choose k}+ \dots+ \lambda_n \cdot {X\choose n}$$  where  $\lambda_0, \dots, \lambda_n$ are integers.
Suppose that $P(k)-P(k')$ is a multiple of $k-k'$ for all  $k, k':=1, \dots, n$. We prove by induction on the degree that $\lambda_k$ is a multiple of $lcm(k)$ for each $k:=1, \dots, n$. 
Let $$Q:= \lambda_0+ \dots+ \lambda_k \cdot  {X\choose k}+ \dots+ \lambda_{n-1}\cdot {X\choose {n-1}}.$$
We have  $Q(k)=P(k)$ for all $k\leq n-1$. Hence, $Q$ satisfies the property,  induction applies and yields  that all $\lambda_k$ are integer  multiples of $lcm(k)$ for $k\leq n-1$. Now, $P(n)= Q(n)+ 
\lambda_{n}\cdot  [{X\choose n}](X= n)$. Since $\lambda_k$ is a multiple of $lcm(k)$ for $k\leq n-1$, it follows from Lemma \ref{wasdifficult} that  $Q$ preserves all congruences, in particular $Q(n)-Q(k)$ is a multiple of $n-k$; since $P(n)-P(k)$ is a multiple of $n-k$,  $P(n)-Q(n)=\lambda_{n}\cdot  [{X\choose n}](X= n)= \lambda_{n}$ is a multiple of $n-k$. Hence $\lambda_n$ is a multiple of $1, \dots, n$. Proving that $\lambda_n$ is a multiple of $lcm({n})$. 
\end{proof}

The proof yields:

\begin{corollary}
If a polynomial of degree $n$ preserves all congruences of the form $\equiv_k$ for $k:=1, \dots n$, it preserves all congruences. 
\end{corollary}

\begin{lemma}\label{chinese} Every map $f$ from a finite subset $A$ of $\Z$ and values in $\Z$ which preserves the congruences extends to every $a\in \Z\setminus A$ to a map with the same property. 
\end{lemma}
\begin{proof}
This  follows from  the Chinese remainder theorem (see Corollary \ref {oneextension} in the next section). 
\end{proof}
Lemma \ref{onepolynomialextension} becomes:
\begin{lemma} $(a)$ Every map $f$ from a finite subset $A$ of $\Z$ and  values in $\Z$ which preserves all congruences extends to a polynomial function preserving all congruences. $(b)$ Every map  $f:\Z :\rightarrow \Z$ which preserves all congruences is of the form  $$\sum_{n=0, \infty} a_{n}\cdot P_{n}$$ where each $a_n$ is an integer multiple of $lcm(n)$.
\end{lemma}
\begin{proof}
We  extend $A$ to a finite  interval $\overline A$. With  Lemma \ref{chinese}, we extend $f$ to $\overline  A$ to a map $\overline f$ which preserves all congruences. The same proof   as in  Lemma \ref {onepolynomialextension} applies. We only need to check that $a_n$ is a multiple of $lcm (n)$ for each $n\in \N$. We do that by induction. We suppose $a_i$ is a multiple of $lcm (i)$ for each $i<n$. We need to prove that $a_n$ is a multiple of $lcm(n)$. The map  $\overline f_{\restriction A_{n+1}}$ preserves the congruences  $\equiv_1, \dots, \equiv_n$, hence, by the proof of Lemma \ref{smalltrick},   $a_n$ is a multiple of $lcm (n)$. 
\end{proof}

Theorem \ref{main} and Theorem \ref{thm:extension} follow.

\section{Further developments}

\subsection{The lattice of congruences}
Let $\Eqv(A)$ be the set of equivalence relations on a set $A$. For a subset $L$ of $\Eqv(A)$,  let $\Pol(L)$, resp. $\Pol^{1} (L)$,  be the set  finitary operations, resp. unary operations,  on $A$ which preserve all members of $R$. Let  $F$ be a set of finitary operations on $A$, the set of equivalence relations preserved by all members of $F$ is the set of congruences of the  algebra  $\mathcal A_F:=(A, F)$, we denote it by  $\Cong(\mathcal A_F)$. 
The pair of maps $L\rightarrow Pol (L)$, $Cong(\mathcal A_F) \leftarrow F$ defines a Galois correspondence between the set $Eq(A)$ of equivalence relations and the set $O_A$ of operations on $A$. This leeds to two problems: 

\noindent 1) Describe the sets of the form $Cong(\mathcal A_F). \\
2)$  Describe the sets of the form $Pol (L)$. 

According to a result of A.Mal'tsev, $\Pol (Cong(\mathcal A_F))$ is determined by its unary part $\Pol^{1}(Cong(\mathcal A_F))$. 

According to  \cite {BKKR69a, BKKR69b}, if $L$ is a set of equivalences on a \emph{finite} set $A$, and $F:= Pol (L)$, then  $\Cong(\mathcal A_F)$   is made of all  equivalence relations definable by primitive positive formulas  from  $L$ (see \cite{snow} for an easy to read presentation). 

If $F$ is a set of maps on $A$, the  set  $\Cong(\mathcal A_F)$ is a subset of $\Eqv(A)$ which is closed under intersection and union of chains. Ordered by inclusion this is an algebraic lattice. It was show by Gr\"atzer and Schmidt \cite{gratzer-schmidt2} that every algebraic lattice is isomorphic to the congruence lattice of some algebra.  

One of the oldest unsolved problem in  universal algebra is "the finite lattice representation problem":

\begin{problem}
Is  every finite lattice  isomorphic to the congruence lattice of  a \emph{finite} algebra? 
(see \cite{palfy1, palfy2}). 
\end{problem}

See \cite{gratzer} for an overview.  The first step in the positive direction is the fact that every finite lattice  embeds as  a sublattice of the lattice of equivalences on a finite set,  a famous and non trivial result of Pudlak and Tuma \cite{pudlak-tuma}. 
Say that a lattice $L$ is \emph{representable} as a congruence lattice if it is isomorphic to the lattice of congruences of some algebra and say that it is \emph{strongly representable} if every  sublattice $L'$ of some $\Eqv(A)$ (with the same $0$ and $1$ elements) which is isomorphic to $L$ is the lattice of congruences of some algebra on $A$. As shown in \cite{quackenbush-wolk}, not every representable lattice is strongly representable. For an integer $n$, let $M_n$ be the lattice made of a bottom and a top element and an $n$-element antichain. Let $M_3$ be the lattice made of a $3$-element antichain and a top and bottom. This lattice is representable (as the set of congruences of the group $\Z/2\cdot\Z\times \Z/2\cdot\Z$) but not strongly representable. We may find sublattices $L$ of $\Eqv (A)$ isomorphic to $M_3$ such that the only unary maps preserving $L$ are the identity and constants. Hence, the congruence lattice of the algebra on $A$ made of these unary maps is $\Eqv (A)$. The sublattices $L$ of $\Eqv(A)$ such that $\Cong(\mathcal A_L)=\Eqv(A)$ (where $\mathcal A_L:=(A,\Pol^{1}(L))$)  are said to be \emph{dense};  the fact that,  as a lattice,  $M_3$ has a dense representation in every $\Eqv(A)$ with $A$ finite on at least five elements, amounting to a Z\'adori's result \cite{Zad83},  appears  in \cite{demeo} as Proposition 3.3.1 on page 20. For an integer n, let $M_n$ be the lattice made of a bottom and a top element and an $n$-element antichain. It is not known if  $M_n$ is representable for each  integer $n$ (it is easy to see that $M_n$ is representable if $n=q+1$ where $q$ is a power of a prime. The case $n=7$ was solved by W.Feit, 1983). Concerning the representability of some type of lattices, note that  every finite distributive lattice is  isomorphic to the lattice of congruences of a finite lattice (Dilworth, \cite{gratzer-schmidt}), hence is representable. In fact, it is  strongly representable  \cite{quackenbush-wolk}. As we will see in Corollary \ref{cor:representation} it is representable as an arithmetic lattice.

\subsection{Arithmetical lattices}

The composition of two binary relation $\theta$ and $\rho$ on  a set $A$ is the binary relation,  denoted by $\theta\circ \rho$,  and defined  by:

$$\theta\circ \rho:= \{(x,y)\in A\times A : (x, z)\in \rho\;  \text {and}\;  (z, y)\in \theta \;  \text {for some}\;  z\in A\}$$
Let $\Eqv(A)$ be he lattice of equivalence relations  on $A$. A sublattice $L$ of $\Eqv(A)$ is \emph{arithmetical} (see \cite {pixley})  if it is distributive and pairs of members of $L$ commute with respect to composition, that is 

\begin{equation}
\rho\circ \theta=\theta \circ \rho\;  \text{for every}\; \theta, \rho\in L.
\end{equation}

This second condition amounts to the fact that the join $\theta\vee \rho$ of $\theta$ and $\rho$ in the lattice $L$ is their composition.

A basic example of arithmetic lattice is the lattice of congruences of $(\Z, +)$. The fact that pairs of congruences commute is easy (and interesting). If $\theta$ and $\rho$ are two congruences, take $(x,y)\in \rho\circ \theta$. Then, there is $z\in \Z$ such that $(x,z)\in \theta$ and $(z,y)\in \rho$. Let $r, t\in \N$ such that $\theta= \equiv_r$ and $\rho= \equiv_t$ , then there are $k,\ell\in \Z$ such that $z=x+k.r$ and $y=z+\ell.t$. Set $z':=x+\ell.t$ then $x\equiv_t z'\equiv_r y$ hence $(x, y)\in \equiv_r\circ \equiv_t= \theta\circ \rho$. Thus $\rho\circ \theta= \theta\circ \rho$ as claimed. 
 
As it is well known, if  $\theta$ and $\rho$ are two congruences,  $\theta= \equiv_{t}$ and $\rho=\equiv_{r}$ with $r, t\in \N$,  then  $\theta\vee \rho=\equiv_{ lcd \{t,r\}}$   
whereas, $\theta\wedge\rho= \equiv_{lcm\{t,r\}}$. Distributivity follows.

As it is well known (see \cite{pixley}), arithmetic lattices can be characterized in terms of the \emph{Chinese remainder condition}.

We say that a sublattice $L$ of $\Eqv(A)$ satisfies the \emph{Chinese remainder condition} if:

for each finite set of equivalence relations $\theta_1, \dots \theta_n$ belonging to $L$ and elements $a_1, \dots, a_n\in A$, the system:
\begin{equation}
x\equiv a_i (\theta_i), i=1, \dots, n
\end{equation} 

is solvable iff for all $1\leq i,j\leq n$

\begin{equation}
a_i\equiv a_j  (\theta_i\vee \theta_j).
\end{equation} 
Recall the following classical result:
\begin{theorem} A sublattice $L$ of $\Eqv(A)$ is arithmetical iff it satisfies  the Chinese remainder condition.
\end{theorem}

Kaarli \cite{kaarli} obtained the following two results:
\begin{corollary}\label {oneextension}
If $L$ is arithmetical (and stable by arbitrary meets) then every partial function $f: B\rightarrow A$ where $B$ is a finite subset of $A$ which preserves all members of $L$ extends to any element $z$ of $A\setminus B$ to a function with the same property. 
\end{corollary}
We recall the proof.
\begin{proof}
Our aim is to find $x\in A$ such that for each $\theta \in L$ and $b\in B$, if $b\equiv z (\theta) $ then  $f(b)\equiv x (\theta)$. Let $B':= f(B)$. For each $b'\in B'$, let $\theta_{b'}$ be the least element of $L$ such that 
\begin{equation} 
b\equiv z(\theta_{b'})
\end{equation}
for all $b$ such that $f(b)= b'$. 

We claim that the system $x \equiv b'(\theta_{b'})$ is solvable and next that any solution yields the element we are looking for.  \end{proof}

\begin{corollary}\label{extension property}
If $L$ is arithmetical on a finite or  countable set $A$, then every partial function $f: B\rightarrow A$ where $B$ is a finite subset of $A$ which preserves all members of $L$ extends to a total function $\overline f$ with the same property. 
\end{corollary}
\begin{proof}Enumerate the elements of $A\setminus B$ in a list $z_0, \dots z_n\dots$. Set $B_n:= B\cup \{z_m:m<n\}$. Define $f_n: B_n\rightarrow A$ in such a way that $f_0= f$ and $f_{n+1}$ extends $f_{n}$ to the element $z_n$ and to no other. Set $\overline f:= \bigcup_n  f_n$. 
\end{proof}

\subsection{Representable lattices  and ultrametric spaces}
\subsubsection{Chinese remainder theorem and metric spaces} 

Chinese remainder condition can be viewed as a property of balls in a metric space. 
For an example,  in the case of $\Z$, if we may view the congruence class of $a_i$ modulo $r_i$ as the (closed) ball $B(a_i, r_i):= \{x\in E: d(a_i, x)\leq r_i\}$ in a metric space $(E, d)$, we are looking for an element of the intersection of these balls. Conditions insuring that such element exists have been considered in metric spaces, Helly property and convexity being the keywords. 
In our case, we may observe that $\Z$ has a structure of ultrametric space, but  the set of values of the distance is not totally ordered. Ordering  $\N$ by the reverse of divisibility: $n\leq m$ if $n$ is a multiple of $m$, we get a (distributive) complete lattice, the least element being $0$, the largest $1$, the  join $n\vee m$ of $n$ and $m$ being the largest common divisor.   Replace the addition by the join and for two elements $a,b\in \Z$, set $d(a,b):= \vert a- b\vert$. Then $d(a,b)=0$ iff $a=b$; $d(a,b)=d(b,a)$ and $d(a,b)\leq d(a,c)\vee d(c,b)$ for all $a,b, c\in \Z$. With this definition, closed balls are congruence classes.  In an ordinary  metric space, a necessary condition for the non-emptiness of the intersection of two balls $B(a_i, r_i)$ and $B(a_j, r_j)$ is that the distance between centers is at most the sum of the radii, i.e. $d(a_i,a_j)\leq r_i+r_j$.  Here this yields $d(a_i,a_j)\leq r_i\vee r_j$ that is $a_i$ and $a_j$ are congruent modulo $lcd(r_i, r_j)$. 
Metric spaces for which this necessary condition suffices are said \emph{convex}. When this condition suffices for the non-emptiness of the intersection of any family of balls they are said \emph{hyperconvex} and \emph{finitely hyperconvex} if it suffices for any finite family. Hence, Chinese remainder theorem of arithmetic is the finite hyperconvexity of $\Z$ viewed as an ultrametric space.   

\subsubsection{Convexity and hyperconvexity}
Let $\mathcal D:=(E,d)$ be a metric space over $\R^{+}$,  $a\in E$ and $r\in \R^{+}$. The \emph{closed ball of center $a$, radius $r$} is the set: 
\begin{equation}
B(a, r): =\{x\in E: d(a,x)\leq r\}.  
\end{equation}
A family $\mathcal H$ of subsets of $E$ has the \emph{$2$-Helly property}  if the intersection of any subfamily $\mathcal H'$ is non-empty provided that the members of $\mathcal H'$ pairwise intersect. We say that $\mathcal H$ has the  \emph{finite $2$-Helly property} provided that the property above holds for finite subfamilies. 

A family $\mathcal B$ of  balls is \emph{convex} if for two members of $\mathcal B$, $B(a, r)\cap B(a',r')\not = \emptyset$ whenever $d(a,a')\leq r+ r'$.

The space $\mathcal D$ is \emph{hyperconvex}, resp. \emph{finitely hyperconvex} if the family of closed balls is convex and has the $2$-Helly property, resp. the finite $2$-Helly property. 

These notions were introduced  in \cite{aronszajn} and developped in \cite{isbell}.  They were extended to metric spaces over a Heyting algebra in \cite{jawhari-all}. The case of ultrametric spaces over a join-semilattice was particularly studied in \cite{pouz-rose}. We present some elements below.

\subsubsection{Ultrametric spaces}

A \emph{join-semilattice} is an ordered set in which two arbitrary elements $x$ and $y$  have  a join, denoted by $x\vee y$,  defined as the least element of the set of common upper bounds of $x$ and $y$.

Let $V$ be  a join-semilattice with a least element, denoted by $0$.
A \emph{pre-ultrametric space} over $V$ is a pair $\mathcal D:=(E,d)$ where $d$ is a map from $E\times E$ into $V$ such that for all $x,y,z \in E$:
\begin{equation} \label{eq:ultra1} 
d(x,x)=0,\; d(x,y)=d(y,x)  \text{~and } d(x,y)\leq d(x,z)\vee d(z, y).
\end{equation}

\noindent The map $d$ is an \emph{ultrametric distance} over $V$ and $\mathcal D$ is an \emph{ultrametric space} over $V$ if $\mathcal D$ is a pre-ultrametric space and $d$ satisfies \emph{the separation axiom}:
\begin{equation} \label{eq:ultra2}
d(x,y)=0\; \text{implies} \;  x=y .
\end{equation}

Any system $\mathcal M:=(E, (\rho_i)_{i\in
I})$ of equivalence relations on  a set $E$ can be viewed as a pre-ultrametric space on $E$. Indeed, 
given a set $I$,  let $\powerset (I)$ be the power set of $I$. Then $\powerset (I)$, ordered by inclusion, is a join-semilattice (in fact a complete Boolean algebra) in which the join is the union, and  $0$  the empty set.
 \begin {proposition}\label{prop:ultra1}
Let $\mathcal M:=(E, (\rho_i)_{i\in
I})$ be a system of equivalence relations. For $x,y\in E$, set $d_{\mathcal M}(x,y):=\{i\in I: (x,y)\not \in \rho_i\}$. Then   the pair $U_{\mathcal M}:=(E, d_{\mathcal M})$ is a pre-ultrametric space over $\powerset (I)$.

\noindent Conversely, let $\mathcal D:=(E,d)$ a pre-ultrametric space over $\powerset (I)$. For every $i\in I$ set $\rho_i:=\{(x,y)\in E\times E: i\not \in d(x,y) \}$ and let  $\mathcal M:=(E, (\rho_i)_{i\in
I})$. Then $\rho_i$ is an equivalence relation on $E$ and $d_{\mathcal M}=d$.

\noindent Furthermore,  $U_{\mathcal M}$ is an ultrametric space if and only if $\bigcap_{i\in I} \rho_i= \Delta_E:=\{(x,x): x\in E\}.$   \end{proposition}

For a join-semilattice $V$  with a $0$ and for two pre-ultrametric spaces $\mathcal D:=(E,d)$ and $\mathcal D':=(E',d')$  over $V$,  a \emph{non-expansive mapping} (or contracting map) from $\mathcal D$ to $\mathcal D'$ is any map $f:E\rightarrow  E'$ such that for all $ x,y\in E$:
\begin{equation}
d'(f(x),f(y))\leq d(x,y). \end{equation}
Pre-ultrametric spaces with  their non-expansive mappings and systems of equivalence relations with their relational homomorphisms are two faces of the same coin.   Indeed: 

\begin{proposition} \label{prop:ultra2} Let $\mathcal M:=(E, (\rho_i)_{i\in
I})$ and $\mathcal M':=(E', (\rho'_i)_{i\in
I})$ be two systems of equivalence relations. A map $f:E\rightarrow E'$ is a  homomorphism from $\mathcal M$ into $\mathcal M'$ if and only if $f$ is a non-expansive mapping from $U_{\mathcal M}$ into $U_{\mathcal M'}$.
\end{proposition}
The  proof  is immediate and left to the reader.
\subsubsection{Metrisation of join-semilattices}
Let $V$ be a join-semilattice with a least element $0$. Let  $d_{\vee}: V\times V \rightarrow V$ defined by $d_{\vee} (x,y)= x\vee y$ if $x\not =y$ and $d_{\vee} (x,y)=0$ if $x=y$. 

\begin{lemma} The map $d_{\vee} $ is a ultrametric distance over $V$ satisfying: 
\begin{equation}\label{eq:maximum}
d_{\vee}(0,x)= x 
\end{equation} 
for all $x\in V$. 

This is the largest ultrametric distance over $V$ satisfying (\ref{eq:maximum}).  

\end{lemma}

\begin{proof}Let $x,y,z$. If two of these elements are equal, the triangular ineqality holds. Otherwise we have trivially $d_{\vee}(x,y)= x\vee y\vee z= d_{\vee}(x,z) \vee d_{\vee}(z,y)$.  This proves that $d_{\vee}$ is an ultrametric distance. If $d$ is any ultrametric distance satisfying (\ref{eq:maximum}) then $d(x,y)\leq d(x,0) \vee d(0,y)=x\vee y$ for every $x,y\in V$. If $x\not =y$ we get $d(x,y)\leq d_{\vee}(x,y)$ and if $x=y$ we get $d(x,y)=0=d_{\vee}(x,y)$. 
\end{proof}

Let  $x,y$ be  two elements of $V$.

If $d$ is any ultrametric distance over $V$  we have:

\begin{equation}
x\leq y\vee d(x,y)
\end{equation}
and 
\begin{equation}
y\leq x \vee d(x,y). 
\end{equation}

This suggests to look at  residuals. 

Let $D(x,y):= \{z\in V: x\leq y\vee z\}$. If $V$ is a distributive lattice then $D(x,y)$ is a filter. 
Indeed, if $x\leq y\vee  z_1$ and $x\leq y \vee z_2$, then $x\leq  (y\vee z_1) \wedge (y \vee z_2)= y \vee (z_1\wedge z_2)$. 
Hence, if  $V$ is finite then $D(x,y)$ has  a least element, the \emph{residual}  of $x$ and $y$. 

In full generality, one defines the \emph{residual} of two elements $x,y$ of  a join-semilattice $V$ (or even a poset) as the least element $x\setminus y$ of the set $D(x,y)$. If $V$ is a Boolean algebra, this is the ordinary  \emph{difference} of $x$ and $y$. We say that $V$ is \emph{residuated} if the residual of any two elements exists.

We say that  a complete lattice $V$ is \emph{$\kappa$-meet-distributive} if for every subset $Z\subseteq V$ with  $\vert Z\vert \leq \kappa$ and $y\in V$,
$$ \wedge  \{y\vee z: z\in Z\}= y\vee \bigwedge Z.$$ It is \emph{completely meet-distributive} if it is $\vert V\vert$-meet-distributive (beware, this terminology has other meanings). 
We have:

\begin{lemma}\label{residual3} Let $V$ be complete lattice. Then   $V$ is residuated if and only if it  completely meet-distributive. \end{lemma}

\begin{proof}
Suppose that $V$ is residuated. Let $y\in V$ and $Z\subseteq V$. Let  $x:= \wedge  \{y\vee z: z\in Z\}$ and let $x\setminus y$ be the residual of $x$ and $y$. 
Trivially $y\vee \bigwedge Z$  is a lower bound of $\{y\vee z: z\in Z\}$. Hence, $y\vee \bigwedge Z\leq  \wedge  \{y\vee z: z\in Z\}=x$. We claim that conversely $x\leq    y\vee \bigwedge Z$. It will follows that 
$\bigwedge \{y\vee z: z\in Z\}= y\vee \bigwedge Z$ as required. Indeed, from the fact that $x$ is a lower bound of $\{y\vee z: z\in Z\}$ we get that $x\setminus y$ is a lower bound of $Z$ and thus $x \setminus y \leq \bigwedge Z$. It follows that $x\leq y \vee x\setminus y\leq y\vee \bigwedge Z$, proving our claim.

 Suppose that  $V$ is complete and completely meet-distributive. Let $x,y\in V$ and $Z:= D(x,y)$. Since $V$ is complete, $\bigwedge Z$ exists. Due to complete meet-distributivity, we have $y\vee \bigwedge Z=\bigwedge \{y\vee z: z\in Z\}\geq x$, hence $\bigwedge Z$ is the least element of $Z$, proving that this is $x\setminus y$.  
\end{proof}

\begin{corollary}\label{cor:distributive1} A finite lattice is residuated iff it is distributive. 
\end{corollary}

\begin{lemma}\label{residual} Let $V$ be  a join semilattice with a least element $0$. 
If the residual of any two elements $x,y$ of $V$ exists, then the map $d_V: V\times V\rightarrow V$ defined by $d_V(x,y):= (x\setminus y) \vee (y\setminus x)$ is  an ultrametric distance over $V$, and in fact the least possible distance $d$ satisfying (\ref{eq:maximum}).

\end{lemma}
\begin{proof}
 Clearly, $d_V(x,y)=0$ iff $x=y$ and $d_V(x,y)=d_V(y,x)$. 
Let $x\in V$. We have $0\setminus x=0$ and $x\setminus 0=x$ hence $d_v(0,x)=x$. 
Let $x,y\in L$. Clearly, $x\leq y\vee (x\setminus y)\leq d_V(x,y)$. Furthermore, $d_V(x,y)\leq x\vee y$. Hence the triangular inequality holds for $\{0, x,y\}$. Now, let $z\in V$. We have:

\begin{equation}\label{eq:residuation}
x\setminus y\leq (x\setminus z) \vee (z\setminus y).
\end{equation}
 Indeed, this inequality amounts to   $x\leq y\vee ( (x\setminus z) \vee (z\setminus y))$. An inequality which follows from the inequalities  $x\leq z\vee (x\setminus z)$ and $z\leq y \vee (z\setminus y)$.   
 
 The triangular inequality follows easily. 
\end{proof}

For an example, if $V$ is a Boolean algebra, $d_V(a,b)= a \Delta b$,  the symmetric difference of $a$ and $b$.

A completely meet-distributive lattice is a special instance of  what we call in \cite{jawhari-all} a Heyting algebra.
From Lemma \ref{residual},  \ref{residual2} and general properties of Heyting algebra presented in \cite {jawhari-all},  we have:
\begin{theorem}
If a join-semilattice $V$ is a Heyting algebra then 
it can be endowed with an ultrametric  distance $d_V$ for which it becomes hyperconvex. Futhermore, every ultrametric metric space over $V$ embeds isometrically into a power of $V$. 
\end{theorem}

 We note that $\N$ ordered by reverse of divisibility is not meet-distributive. Still, it can be equipped with a distance (given by the absolute value). It is finitely hyperconvex, but it is not hyperconvex. Indeed, an infinite set of equations  does not need to have a solution while every  finite subset has one (for an example, let $a_{2n}:= 2$, $r_{2n}:= 2^{n}$, $a_{2n+1}:=3$, $r_{2n+1}:= 3^{n}$, then $d(a_{2n}, a_{2m})=0\leq r_{2n}\vee r_{2m}$, $d(a_{2n}, a_{2m+1})= 1\leq r_{2n}\vee r_{2m+1}= lcd(2^n, 3^m)=1$).   

This suggests to look at  possible extensions of the results obtained so far about metric spaces over meet-ditributive lattices. We recall some of them.

Let $\mathcal D:=(E,d)$ be a metric space over a join-semilattice $V$. For each $r\in V$ set $\equiv_r:= \{(x,y)\in E: d(x,y)\leq r\}$. Let $Eq_{d}(E):= \{\equiv_r: r\in V\}$.  Let $F:=Hom((E,d), (E,d))$ be the set of contracting maps from $(E,d)$ into itself, let  $\mathcal E_F= (E, F)$ be the algebra made of  unary operations $f\in F$ and let  $Cong_{d}(E):= Cong(\mathcal E_F)$ be the set of congruences of this algebra, that is  the set of all equivalence relations on $E$ preserved by all contractions from $E$ into itself.

\begin{proposition}\label{fourre-tout} Let $(E,d)$ be an ultrametric space over a join-semilattice $V$ with a least element $0$. Then:
\begin{enumerate}

\item $Eq_{d}(E) \subseteq Cong_{d}(E)$. 

\item A map $f:E \rightarrow E$ is a contraction of $(E, d)$ into itself iff  it preserves all members of $Eq_{d}(E)$. 

\item If the meet of every non-empty subset of $V$ exists, then $Eq_d(E)$ is an intersection closed subset of $Eq(E)$,  the set of equivalence relations on $E$. 
\item The set $Cong_{d}(E)$ is an algebraic lattice; furthermore, for every $(x,y) \in E \times E$,  the  least member   $\delta(x,y)$ of $Cong_{d}(E)$ containing $(x,y)$ is a compact element of $Cong_{d}(E)$. Furthermore, $\delta(x,y)$ is included into $\equiv_r$, where   $r:= d(x,y)$. 

\item Any two members of $Eq_d(E)$ commute and $\equiv_r\circ \equiv_s=\equiv_s\circ \equiv_r= \equiv_{r\vee s}$ for every $r,s\in  V$  iff $(E,d)$ is convex. 
\item $(E, d)$ is hyperconvex iff $Eq_d(E)$ is a completely meet-distributive lattice of $Eq(E)$. 
\end{enumerate}
\end{proposition}
\begin{proof}The first two item are  immediate. Trivially,  each $\equiv_r$ is an equivalence relation and it  is preserved by all contracting maps.  Item (3). Let $\equiv_{r_i}$, $i\in I$ be a family of members of  $Eq_d(E)$ then 
$\bigcap_{i\in I} \equiv_{r_i}$ equals $\equiv_r$ where $r:= \bigwedge \{r_i: i\in I\}$. Item (4). Since $(x,y)\in \equiv_r$ and $\equiv_r$ is preserved by all contractions, we have $\delta(x,y)\subseteq \equiv_r$. Since $Cong_{d}(E)$ is the  congruence lattice of an algebra it is algebraic. The fact that $\rho(x,y)$ is algebraic follows from the algebraicity of $Cong_{d}(E)$. 
 Item (5) is  Proposition 3.6.7 of \cite {pouz-rose}. We recall the proof.  Let $r,s\in V$. Due to the triangular inequality, we have $\equiv_s\circ \equiv r \subseteq \equiv _{r\vee s}$. We claim that the equality holds whenever  $(E,d)$ is convex. Let $t:= r\vee s$ and  $(x,y)\in \equiv_t$. Since $d(x,y)\leq t= r\vee s$ and $(E,d)$ is convex, the closed balls $B(x, r)$ and $B(y, s)$ intersect. If  $z$ belongs to this intersection, then $d(x,z)\leq r$ and $d(y, s)\leq  t$ hence $(x,y)\in \equiv_s\circ \equiv_r$. This proves our claim. Conversely, let $B(x, r)$ and $B(y, s)$ with $d(x,y)\leq r\vee s$, that is $(x,y)\in \equiv_{r\vee s}$. We have $\equiv_r \vee \equiv_ s \subseteq \equiv_{r \vee s}$ and since $r$ and $s$ commute, $\equiv_{s}\circ \equiv_r=  \equiv_r \vee \equiv_s$. Due to our assumption $\equiv_{r} \vee \equiv_{s}= \equiv_{r \vee s}$, hence $\equiv_r \vee \equiv_s$ is the join in $ Eq_d(E)$;  furthermore, since  $\equiv_{s}\circ \equiv_r=  \equiv_r \vee \equiv_s$ there is some $z\in E$ such that $z\in  B(x, r)\cap B(y, s)$. 
Item (6) is Proposition 3.6.12 of \cite {pouz-rose}.
\end{proof}

\begin{corollary} If $(E,d)$ is convex the map $r\rightarrow \equiv_r$ is  a lattice homomorphism from $V$ into $Eq_d(E)$.  
\end{corollary}

\begin{theorem}\label{thm:hyperconvex} If an ultrametric space  $(E,d)$  is hyperconvex, then every member of  $Cong_{d}(E)$ is a join of equivalence relations of the form $\equiv_r$, for $r\in V$. 
\end{theorem}

\begin{proof} Let $\rho$ be an equivalence relation on $E$. Let $(x,y)\in \rho$ and $r:= d(x,y)$. We claim that if $\rho$ is preserved by every contracting map then $\equiv_r\subseteq \rho$. Indeed, let $(x',y') \in \equiv_r$. The (partial) map $f$ sending $x$ to $x'$ and $y$ to $y'$ is contracting. Since $(E,d)$ is hyperconvex, it extends to $E$ to a contracting map $\overline f$. Since $\rho$ must be preserved by $\overline f$, and $(x,y)\in \rho$,  we have $(x',y')\in \rho$. This proves our claim. From item (4) of Proposition \ref{fourre-tout} it follows that $\delta(x,y)= \equiv_r$. Also,  $\rho$ is the union of all $\equiv_r$ it contains. 
\end{proof}

\begin{lemma}\label{residual2}  $V$ is algebraic, the residual of two compact elements is compact. 
\end{lemma}
 \begin{proof} Suppose $x$ and $y$ compact. Suppose $x\setminus y\leq \bigvee Z$ for some subset $Z$ of $V$.  We have $x \leq y \bigvee Z$. Since $x$ is compact, $x\leq y \bigvee Z'$ for some finite $Z'\subseteq Z$. Since $x\setminus y$ is the least $z$ such that $x\leq y\vee z$, we have $x\setminus y\leq \bigvee Z'$ proving that $x\setminus y$ is compact. 
\end{proof}

\begin{theorem} Let $L$ be an algebraic lattice and $K(L)$ be the join-semilattice of compact elements  of $L$. If $L$ is meet-distributive then $K(L)$ has an ultrametric structure 
and  $L$ is isomorphic to the set of equivalence relations on  $K(L)$  preserved by all contracting maps on $K(L)$. 
\end{theorem} 

\begin{proof} Due to Lemma \ref{residual}, we may define on  $V:= K(L)$ the distance $d_V$. Due to meet-distributivity, $V$ is hyperconvex. According to Theorem \ref{thm:hyperconvex} each equivalence relation preserved by all contracting operation is a join of equivalence relations of the form $\equiv_r$ for some $r:=d_V(a,b)$. 
\end{proof}
\begin{corollary} If $V$ is a finite distributive lattice, then $V$ is isomorphic to the lattice of equivalence relations preserved by all contracting maps from $V$ into itself, $V$ being equipped with the distance $d_V$. 
\end{corollary} 

In guise of conclusion, we get:

\begin{corollary} \label{cor:representation}A finite  lattice is distributive iff it is representable as an arithmetical lattice.
\end{corollary}

\noindent {\bf Bibliographical comment. }
The study of metric spaces with distance values in a Boolean algebra  first appears in Blumenthal \cite {blum}.  A  study of metric spaces with distance   values in a Heyting algebra is in \cite{jawhari-all}; a more general study   is in \cite{pouz-rose}. Ultrametric spaces with   distance values in an ordered set have been studied by Priess-Crampe and Ribenboim in  several papers, e.g.  \cite{ribenboim1}, \cite{ribenboim2},\cite{ribenboim3}. \\

\noindent {\bf Acknowledgements.} We are pleased to thank Serge Grigorieff for discussions    about CGG's theorem and bibliographical informations (notably \cite{bhargava}, \cite{polya}).


\begin{thebibliography}{99}
\vspace*{-1em}
\normalsize
\bibitem{aronszajn} Aronszajn, N.;   Panitchpakdi, P.,    Extension of uniformly continuous transformations and hyperconvex metric spaces, Pacific J. Math. 6 (1956) 405-439. 
\bibitem{bhargava} 
    Bhargava, M.,  The factorial function and generalizations,
   Amer. Math. Monthly. 107(2000), 783-799.

\bibitem
{blum} Blumenthal, L.M., Theory and applications of distance geometry. Oxford, at the Clarendon Press, 1953. xi+347 pp.

\bibitem
{BKKR69a} Bodnar\u{c}uk, V.G.; Kalu\u{z}nin, L.A.; Kotov,
		 N.N.; Romov, B.A., Galois theory for Post algebras
		 I, {\sl Kibernetika} {\boldmath 3}, 1-10, 1969
	(Russian). English translation {\sl Cybernetics and Systems Analysis}, {\boldmath 5}, 1969.
\bibitem
{BKKR69b} Bodnar\u{c}uk, V.G.; Kalu\u{z}nin, L.A.; Kotov, N.N.; Romov, R.A., Galois theory for Post algebras
		 II, {\sl Kibernetika} {\boldmath 5}, 1-9, 1969
	(Russian). English translation {\sl  Cybernetics and Systems Analysis}, {\boldmath 5},1969.
	\bibitem{cgg1} C\'egielski, P.;  Grigorieff, S.; Guessarian, I.,  Integral difference ratio functions on integers. Computing with new resources, 277-291, Lecture Notes in Comput. Sci., 8808, Springer, Cham, 2014.
	\bibitem{cgg2} C\'egielski, P.; Grigorieff, S.; Guessarian, I.,  Newton representation of functions over natural integers having integral difference ratios. Int. J. Number Theory 11 (2015), no. 7, 2109--2139.

%
\bibitem{DLPS} Delhomm\'e, C.; Miyakawa, M.; Pouzet, M.;  Tatsumi H., Semirigid systems of three equivalence relations, to appear in MVLSC, arXiv:1505.02955 .
\bibitem{delhomme1} C. Delhomm\'e, Personnal communication, November 29, 2016.

\bibitem{delhomme2} Delhomm\'e, C., Arithm\'etique et ultram\'etrique g\'en\'eralis\'ee, .pdf, 85 pp. December 31, 2016.

\bibitem{demeo} DeMeo, W.J., Congruence lattices of finite algebras, PhD, May 2012, University of Hawaii at Manoa, arXiv:1204. 4305v3[math.GR]. 
%
%
%
%
\bibitem {gratzer-schmidt}Gr\"atzer, G.; Schmidt, E.T.,
On congruence lattices of lattices.
Acta Math. Acad. Sci. Hungar. 13 (1962) 179--185. 
 \bibitem {gratzer-schmidt2} Gr\"atzer, G;  Schmidt,E.T.,  Characterizations of congruence lattices of abstract algebras. Acta Sci. Math. (Szeged) 24 (1963) 34--59.
  \bibitem {gratzer-schmidt3} Gr\"atzer,G.;   Schmidt, E.T.,  Finite lattices and congruences. A survey. Algebra Universalis 52 (2004), no. 2-3, 241--278. 
  \bibitem{gratzer} Gr\"atzer, G.,Two problems that shaped a century of lattice theory, Notices of the AMS, 54,6 (2007) 296--707.
  \bibitem{isbell} Isbell, J., Six theorems about injective metric
spaces, Comment. Math. Helv. 39 (1964), 65-76.
\bibitem{jawhari-all}Jawhari, El M.; Pouzet, M.; Misane, D., Retracts: graphs and ordered sets from the metric point of view. Combinatorics and ordered sets (Arcata, Calif., 1985), 175--226, Contemp. Math., 57, Amer. Math. Soc., Providence, RI, 1986. 
\bibitem{kaarli}Kaarli, K., Compatible function extension property, Algebra Univers. 17 (1983), 200--207.
\bibitem{kaarli}Kaarli, K., 
A new characterization of arithmeticity. 
Conference on Lattices and Universal Algebra (Szeged, 1998).
Algebra Universalis 45 (2001), no. 2-3, 345--347. 
 \bibitem{kaarli-pixley} Kaarli, K; Pixley, A F.,  Polynomial completeness in algebraic systems. Chapman and  Hall/CRC, Boca Raton, FL, 2001. xvi+358 pp. 
 

 \bibitem{palfy1} P\'alfy, P.; Pudl\'ak,P.,   Congruence lattices of finite algebras and intervals in subgroup lattices of finite groups. Algebra Universalis 11(1), 22--27 (1980). 
%
\bibitem{palfy2} P\'alfy, P.,   Groups and lattices. In Groups St. Andrews 2001 in Oxford. Vol. II, volume 305 of London Math. Soc. Lecture Note Ser., pages 428--454, Cambridge, 2003. Cambridge Univ. Press.
%
\bibitem{pixley}Pixley, A. F.,  Functional and affine completeness and arithmetical varieties. Algebras and orders (Montreal, PQ, 1991), 317?357, NATO Adv. Sci. Inst. Ser. C Math. Phys. Sci., 389, Kluwer Acad. Publ., Dordrecht, 1993. 


  \bibitem {polya} Polya, G.,  Uber ganzwertige ganze Funktionen,
  Rend. Circ. Mat. Palermo 40  (1915), 1-16.
  
  

\bibitem{pouz-rose} Pouzet, M.; Rosenberg, I.G., General metrics and contracting operations. Graphs and combinatorics (Lyon, 1987; Montreal, PQ, 1988). Discrete Math. 130 (1994), no. 1-3, 103-169.
\bibitem{pudlak-tuma}Pudl\'ak, P.;   Tuma, J., Every finite lattice can be embedded into a finite partition lattice, Algebra Universalis 10 (1980), 74--95. 
\bibitem{quackenbush-wolk} Quackenbush,R.; Wolk,B.,  Strong Representation of Congruences Lattices, Algebra Universalis, 1 (1971) 165--166. 
\bibitem{ribenboim1}Priess-Crampe, S.; Ribenboim, R., Generalized ultrametric spaces I, Abh. Math. Sem. Univ. Hamburg 66 (1996) 55--73.
\bibitem{ribenboim2}Priess-Crampe, S.; Ribenboim, R., Generalized ultrametric spaces II, Abh. Math. Sem. Univ. Hamburg 67 (1997) 19--31. 
\bibitem{ribenboim3} Priess-Crampe, S.;  Ribenboim, R., Homogeneous ultrametric spaces, J. Algebra 186 (1996) 401--435. 

\bibitem{snow}Snow, J.W,  A constructive approach to the finite congruence lattice representation problem. Algebra Universalis 43 (2000), no. 2-3, 279--293.
\bibitem [Zad83]
{Zad83} Z\'adori,  L., Generation of finite partition
 lattices, {\it Lectures in Universal Algebra} (Proc. Colloq., Szeged,
1983),  Colloq. Math. Soc. J\'anos Bolyai 43, North-Holland,
Amsterdam, 1986, 573-586.
\end{thebibliography}
\end{document}